\documentclass[a4paper,12pt]{amsart}
\usepackage[left=2cm,top=2cm,right=2cm,bottom=2cm]{geometry}
\usepackage{amsmath}
\usepackage{amssymb}
\usepackage{amsfonts}
\usepackage{amsthm}

\usepackage{stmaryrd}

\newtheorem{thm}{Theorem}[section]
\newtheorem{lemma}[thm]{Lemma}
\newtheorem{prop}[thm]{Proposition}
\newtheorem{cor}[thm]{Corollary}

\theoremstyle{definition}

\theoremstyle{definition}
\newtheorem{defn}[thm]{Definition}

\newcommand{\supp}{\ensuremath{\operatorname{supp}}}

\newcommand{\norm}[1]{{\left\|{#1}\right\|}}

\newcommand{\abs}[1]{{\left|{#1}\right|}}
\newcommand{\scal}[1]{{\left\langle{#1}\right\rangle}}
\newcommand{\set}[1]{{\left\{{#1}\right\}}}

\newcommand{\al}{\ensuremath{\alpha}}
\newcommand{\be}{\ensuremath{\beta}}
\newcommand{\ga}{\ensuremath{\gamma}}
\newcommand{\Ga}{\ensuremath{\Gamma}}
\newcommand{\de}{\ensuremath{\delta}}

\newcommand{\ze}{\ensuremath{\zeta}}
\newcommand{\ka}{\ensuremath{\kappa}}
\newcommand{\La}{\ensuremath{\Lambda}}
\newcommand{\la}{\ensuremath{\lambda}}
\newcommand{\sig}{\ensuremath{\sigma}}
\newcommand{\vphi}{\ensuremath{\varphi}}
\newcommand{\om}{\ensuremath{\omega}}

\newcommand{\bears}{\begin{eqnarray*}}
\newcommand{\eears}{\end{eqnarray*}}

\newcommand{\mc}[1]{\ensuremath{\mathcal{#1}}}

\newcommand{\ZZ}{\ensuremath{\mathbb{Z}}}

\newcommand{\QQ}{\ensuremath{\mathbb{Q}}}
\newcommand{\RR}{\ensuremath{\mathbb{R}}}

\newcommand{\CC}{\ensuremath{\mathbb{C}}}

\newcommand{\sm}{\ensuremath{\setminus}}
\newcommand{\ssq}{\ensuremath{\subseteq}}

\newcommand{\rar}{\ensuremath{\rightarrow}}
\newcommand{\longrar}{\ensuremath{\longrightarrow}}

\newcommand{\vn}{\ensuremath{\varnothing}}

\newcommand{\bs}[1]{\ensuremath{\mathbf{#1}}}
\providecommand{\abs}[1]{\lvert#1\rvert}

\numberwithin{equation}{section}

\title{A note on Gabor frames in finite dimensions}

\author{Romanos-Diogenes Malikiosis}
\date{}

\begin{document}

\begin{abstract}
 The purpose of this note is to present a proof of the existence of Gabor frames in general linear position
 in all finite dimensions. The tools developed in this note are also helpful towards an explicit construction
 of such a frame, which is carried out in the last section. This result has applications in signal recovery
 through erasure
 channels, operator identification, and time-frequency analysis.
\end{abstract}

\maketitle

 \section{Introduction}
 
 A Gabor frame is the set of all time-frequency translates of a single vector in $\CC^N$, and consists of
 $N^2$ vectors. The question that we will tackle in this paper, is whether any subset of $N$ vectors is
 linearly independent. In this case we shall say that the Gabor frame is in general linear position.
 
 The existence of Gabor frames in general linear position in all dimensions has very deep implications in
 signal processing; such a frame is an \emph{equal norm tight frame that is maximally robust to erasures},
 a fact which allows the recovery of the original object even if many packets of encrypted information are lost.
 The only previously known frames with this property are the harmonic frames. This problem may also be viewed as
 the discrete version of the HRT conjecture \cite{HRT96}, which asserts that any finite set
 of time-frequency translates of a nonzero function in $L^2(\RR)$ is linearly independent. 
 
 In the next two subsections of the introduction, we present the 
 main result, as well as some basic consequences, in order to make this note self-contained. However, the 
 focus will be on the proof of the main result, so for completeness we refer the reader to \cite{KPR08,LPW05,P13},
 and the references within, which provide more details regarding applications. Besides,
 \cite{P13} offers an excellent introduction on the subject of Gabor frames in finite dimensions.
 
 Then, the paper is organized as follows: in Section 2 we present the idea behind the proof in \cite{LPW05};
 sections 3 and 4 form the proof of the
 main result; and lastly, we present a construction of a Gabor frame in general linear position in the
 last section.

 \subsection{Setup and main result}
 
 We define the following two linear operators of $\CC^N$: the \emph{cyclic shift} operator
 $T:\CC^N\rar\CC^N$, which is given by
 \[T(x_0,x_1,\dotsc,x_{N-1})=(x_{N-1},x_0,\dotsc,x_{N-2}),\]
 and the \emph{modulation} operator $M:\CC^N\rar\CC^N$, which is given by
 \[M(x_0,x_1,\dotsc,x_{N-1})=(x_0,\om x_1,\dotsc,\om^{N-1}x_{N-1}),\]
 where $\om=e^{2\pi i/N}$. These two operators generate a group called the \emph{Weyl-Heisenberg group},
 otherwise called the \emph{generalized Pauli group}. The relation $MT=\om TM$, shows that the Weyl-Heisenberg
 group modulo phases is isomorphic to $(\ZZ/N\ZZ)^2$, and a complete set of representatives
 is given by $\pi(\ka,\la)=M^{\la}T^{\ka}$, for $(\ka,\la)\in(\ZZ/N\ZZ)^2$ (for more information regarding
 the algebraic structure of the Weyl-Heisenberg group in relation to a problem of similar nature (SIC-POVM), we refer
 the reader to \cite{A05}).
 
 \begin{defn}
  A Gabor frame $(\vphi,\La)$ with $\vphi\in\CC^N$ and $\La\ssq(\ZZ/N\ZZ)^2$, is the set of all vectors of
  the form $\pi(\ka,\la)\vphi$, where  $(\ka,\la)\in\La$. We say that $(\vphi,\La)$ is in general linear position, if
  $(\vphi,\La')$ is a basis for every $\La'\ssq\La$ with $\abs{\La'}=N$.
 \end{defn}
 
 In this note, we shall prove the following:
 
 \begin{thm}\label{main}
  For every positive integer $N$, there is some $\vphi\in\CC^N$, such that $(\vphi,(\ZZ/N\ZZ)^2)$ is in general
  linear position. Moreover, the set of such $\vphi$ is of full measure; its complement has Lebesgue measure zero.
 \end{thm}
 
 So far, this theorem has been proven for $N$ prime \cite{LPW05}, and numerical solutions have been
 given for $N=4,6$ \cite{KPR08} and $N=8$ \cite{DBBA12}.
 
 \subsection{Consequences}

 We need some definitions first:
 
 \begin{defn}
  Let $N$ be a positive integers, and $\set{\vphi_k}_{k\in K}$ a finite set of vectors in $\CC^N$. If the
  inequalities
  \[c_1\norm{f}_2^2\leq\sum_{k\in K}\abs{\scal{f,\vphi_k}}\leq c_2\norm{f}_2^2\]
  are true for all $f\in\CC^N$, for some $0<c_1\leq c_2$, then $\set{\vphi_k}_{k\in K}$ is called a
  \emph{frame} for $\CC^N$. It is called \emph{tight}, if we can take $c_1=c_2$, and if $\norm{\vphi_k}_2=C>0$ for all
  $k\in K$, then it is called an \emph{equal norm tight frame}. If any subset of $\leq N$ vectors in 
  $\set{\vphi_k}_{k\in K}$ is linearly independent, then we shall say that the frame is \emph{maximally
  robust to erasures}.
 \end{defn}
 
 The set of vectors $(\vphi,(\ZZ/N\ZZ)^2)$ is indeed an equal norm tight frame \cite{LPW05,P13}. 
 A major consequence of Theorem \ref{main} is the fact that this
 Gabor frame is also maximally robust to erasures; actually, this
 is true for all vectors $\vphi$, except for a those belonging to a set of Lebesgue measure zero.
 
 
 \begin{defn}
  Let $\mc{H}$ be a linear space of operators, mapping $\CC^N$ to $\CC^M$. $\mc{H}$ is called \emph{identifiable}
  with identifier $\vphi\in\CC^N$, if the map $H\mapsto H\vphi$ from $\mc{H}$ to $\CC^M$ is injective.
 \end{defn}
 
 We will denote by $\mc{H}_{\La}$ the (complex) linear space of operators that is spanned by $\pi(\ka,\la)$,
 $(\ka,\la)\in\La$.
 
 \begin{defn}
  The short-time Fourier transform $V_{\vphi}:\CC^N\longrightarrow\CC^{N^2}$ with respect to the \emph{window}
  $\vphi\in\CC^N$ is given by
  \[V_{\vphi}f(\ka,\la)=\scal{f,\pi(\ka,\la)\vphi},\]
  for all $f\in\CC^N$, $\ka,\la\in\ZZ/N\ZZ$. Denote by $A_{\vphi}$ the matrix representation of $V_{\vphi}$,
  under some ordering of $(\ZZ/N\ZZ)^2$.
 \end{defn}
 
 Theorem \ref{main} and Theorem 5.7 in \cite{KPR08} yield the following theorem.
 
 \begin{thm}
  Let $N$ be a positive integer. Then, for almost all $\vphi\in\CC^N$, the following equivalent conditions hold.
  \begin{enumerate}
   \item Every minor of $A_{\vphi}$ of order $N$ is nonzero.
   \item The Gabor frame $(\vphi,(\ZZ/N\ZZ)^2)$ is in general linear position.
   \item The Gabor frame $(\vphi,(\ZZ/N\ZZ)^2)$ is an equal norm tight frame that is maximally robust to 
   erasures.
   \item For all $f\in\CC^N\sm\set{\bs 0}$, we have $\abs{\supp(V_{\vphi}f)}\geq N^2-N+1$.
   \item For all $f\in\CC^N\sm\set{\bs 0}$, $V_{\vphi}f$, and therefore $f$ is completely determined by its 
   values on a set $\La$ with $\abs{\La}=N$.
   \item $\mc{H}_{\La}$ is identifiable by $\vphi$ if and only if $\abs{\La}=N$.
  \end{enumerate}
 \end{thm}
 
 Construction of such a vector $\vphi$, satisfying all of the above conditions is accomplished in Section
 \ref{construct}.

 As mentioned above, Theorem \ref{main} shows us that there is a way to recover encrypted signals through
 erasure channels, even if we lose a great amount of information packets. This is accomplished in the
 following way; if $\set{\vphi_k}_{k\in K}$ is
 a Gabor frame in general linear position, we encode information in the form of a vector
 $f\in \CC^N$ as follows: we send through a channel the Hermitian inner products $\scal{f,\vphi_k}$, and
 we assume that this is a channel with \emph{erasures}, which means that some of these products could be
 lost (but the recipient knows which indices $k$ correspond to these lost products). The recipient 
 receives $\scal{f,\vphi_k}$, where $k\in K'\ssq K$. Can we reconstruct $f$ from this information? The answer
 is yes, as long as $\abs{K'}\geq N$, by finding a dual frame $\set{\tilde{\vphi}_k}_{k\in K'}$, and then use
 the formula
 \[f=\sum_{k\in K'}\scal{f,\vphi_k}\tilde{\vphi}_k.\]
 The reason that this reconstruction is possible follows from the fact that any $N$ vectors from 
 the Gabor frame $\set{\vphi_k}_{k\in K}$ are linearly independent.

 \section{Summary of the proof for $N$ prime}
 
 For any $\La\ssq(\ZZ/N\ZZ)^2$ with $\abs{\La}=N$, we let
 the operators $\pi(\ka,\la)$ act on the variable vector $z=(z_0,\dotsc,z_{N-1})$, for $(\ka,\la)\in\La$. The 
 coordinates of the vectors form a $N\times N$ matrix, so $(z,\La)$ forms a basis if and only if the
 determinant of this matrix is nonzero. As $z$ is a variable vector, the determinant is a homogeneous
 polynomial in $z_0,\dotsc,z_{N-1}$, so we ask whether this polynomial is identically zero or not.
 Lawrence, Pfander, and Walnut \cite{LPW05} proved that every such polynomial is nonzero; therefore, 
 the zero set has
 Lebesgue measure zero, and the union of the zero sets of all such polynomials is still of Lebesgue measure
 zero, because they are finitely many. So, any vector not belonging to this union, say $\vphi$, forms a
 Gabor system $(\vphi,(\ZZ/N\ZZ)^2)$ in general linear position.
 
 
 It is important to analyze how the authors in \cite{LPW05} proved that such a polynomial is nonzero; they isolated a 
 certain monomial, and then showed that its coefficient is a product of minors of the Fourier matrix
 up to a phase. Then, by Chebotarev's theorem (every minor of the Fourier matrix in dimension $p$
 is nonzero, for prime $p$) we can deduce that this coefficient is nonzero. In fact, they proved something
 stronger: that the determinant of every submatrix of the $N\times N^2$ matrix that is formed by the
 column vectors  $\pi(\ka,\la)z$ is a nonzero polynomial. 
 
 
 Let $D$ be such a $N\times N$ submatrix.
 We define the monomial $p_D$ which is obtained as follows: if $N=1$, take $p_D$ to be the
 only variable that appears in $D$. If $N>1$, take the variable that appears in $D$ with minimal index
 (i.~e. least value for its subscript),
 which will be $z_0$,
 then erase the column and row that correspond in this entry and repeat the process for the $(N-1)\times (N-1)$
 submatrix obtained this way. Define by $p_D$ the product of all these variables. Even though the choice
 of an entry with minimal index might not be unique, it turns out that $p_D$
 is well defined, i.~e. it is independent of the choice of variable at every step, as long as the variable
 we choose at each step
 has minimal index. For more details, we refer the reader to \cite{LPW05}.
 
 
 The monomial obtained this way, shall be called the \emph{lowest index monomial}, for the following reason: if
 we list all monomials that appear in the formal expansion of the determinant of $D$, then $p_D$ is the first
 in alphabetical order, given the ordering $z_{j-1}<z_j$, $1\leq j\leq N-1$, assuming that we write the 
 variables of each monomial in increasing order. Every monomial can be obtained through the diagonals of $D$;
 a diagonal of $D$ is simply a set of $N$ entries of $D$, no two of which lie in the same row or column.
 Clearly, $D$ has $N!$ diagonals.
 
 
 We associate to matrix $D$ the $N$-tuple $(l_0,l_1,\dotsc,l_{N-1})$, where $l_{\ka}$ is the number of columns in $D$
 of the form $\pi(\ka,\la)z$. After rearranging the columns (which could only change the sign of the determinant),
 we can write 
 \[D=(D_0|D_1|\dotsb|D_{N-1}),\]
 so that all columns of $D$ of the form $\pi(\ka,\la)z$ form the $N\times l_{\ka}$ submatrix $D_{\ka}$
 (it is understood that the columns in $D_{\ka}$ are written in increasing order, in terms of $\la$). By definition,
 \[l_0+l_1+\dotsb+l_{N-1}=N.\]
 We label the rows of $D$ by $0,1,\dotsc,N-1$, the $0$th row being the top one, and the $(N-1)$th row being the bottom one.
 Consider an ordered partition of $\ZZ/N\ZZ$ into sets $B_0,B_1,\dotsc,B_{N-1}$, such that $\abs{B_{\ka}}=l_{\ka}$. 
 We denote by
 $D_{\ka}(B_{\ka})$ the $l_{\ka}\times l_{\ka}$ submatrix of $D_{\ka}$, whose rows belong to the set $B_{\ka}$, when $B_{\ka}\neq\vn$.
  From this
 construction, it is clear that there are $l_0!l_1!\dotsm l_{N-1}!$ diagonals for which each element belongs to either one of the
 $D_{\ka}(B_{\ka})$, and they all give rise to the \emph{same} monomial that appears in $\det(D)$. 
 It is also evident that any monomial can be obtained by $K\prod_{\ka=0}^{N-1}l_{\ka}!$ diagonals, where $K$ 
 is a nonnegative integer.
 In other words, any partition of $\ZZ/N\ZZ$ gives rise to a monomial; we will say that a monomial appears 
 \emph{uniquely} in $D$, if it corresponds to a unique partition of $\ZZ/N\ZZ$.
 
 
 However, for any such partition there is some $\sig\in S_N$ such that $\sig(A_{\ka})=B_{\ka}$, for all $\ka$, where 
 \begin{equation}\label{partition}
 \begin{split}
A_0 &= \set{0,1,\dotsc,l_0-1}\\
A_1 &= \set{l_0,l_0+1,\dotsc,l_0+l_1-1}\\
&\	\vdots\\
A_{N-1} &= \set{l_0+l_1+\dotsb+l_{N-2},\dotsc,N-1}.
\end{split}
\end{equation}
In other words, any permutation $\sig\in S_N$ gives rise to a monomial in $\det(D)$, which we will denote by $Z^{\sig}$. From this
definition, it is obvious that if $\tau\in S_N$ leaves all the sets $A_0,\dotsc,A_{N-1}$ invariant, then $Z^{\sig}=Z^{\sig\tau}$.
We will call such a permutation \emph{trivial}, and denote the subgroup of all trivial permutations by $\Ga$, so that the map
\[S_N/\Ga\ni\sig\mapsto Z^{\sig}\]
is well defined. These definitions
yield
\[\det(D)=\sum_{\sig\in S_N/\Ga}c_{\sig}Z^{\sig},\]
and the previous discussion implies that
\[c_{\sig}Z^{\sig}=\pm\prod_{\ka=0}^{N-1}\det(D_{\ka}(\sig(A_{\ka}))),\]
where $\det(D_{\ka}(B_{\ka}))=1$ when $B_{\ka}=\vn$.
So, a monomial $Z$ in $\det(D)$ appears uniquely, if there is a unique $\sig\in S_N/\Ga$, such that $Z=Z^{\sig}$.
 
 The authors of \cite{LPW05} proved that the lowest index monomial appears uniquely; it turns out that the coefficient
  is a product of Fourier minors, up to phase, so when $N$ is prime, this coefficient is nonzero, as follows
 from Chebotarev's theorem.

 \section{The consecutive index monomial}\label{CI}
 
 In the general case for $N$, the lowest index monomial is still obtained uniquely; however, not all Fourier minors are nonzero
 when $N$ is composite, so it might appear with coefficient zero. For this reason, we will try to focus on another monomial.
 
 \begin{defn}\label{CIM}
 The \emph{consecutive index monomial} (CI monomial for short) is the monomial that corresponds to the partition
 $A_0,A_1,\dotsc,A_{N-1}$ described in \eqref{partition}. Equivalently, it is the monomial $Z^{\iota}$, where
 $\iota$ is the the identity permutation of $\ZZ/N\ZZ$. We will simply denote this monomial by $Z$.
 \end{defn}
 
 As we will prove later,  the indices of the variables appearing in the CI monomial are consecutive, when viewed
 as elements of $\ZZ/N\ZZ$; this means that $N-1$ are $0$ are considered consecutive elements, so for example, $z_0z_1z_{N-1}^{N-2}$
 is a monomial whose indices are consecutive.
 
 
 \begin{prop}\label{unique}
 If the CI monomial appears uniquely in $D$, then its coefficient in $\det(D)$ is nonzero.
 \end{prop}
 
 \begin{proof}
 This monomial appears in $\det(D)$ as
 \[\prod_{\ka=0}^{N-1}\det(D_{\ka}(A_{\ka})),\]
 where the $A_{\ka}$ are given by \eqref{partition}. Let 
 $\pi(\ka,\la_i)z$ be the columns of $D_{\ka}$, for $1\leq i\leq l_{\ka}$. 
 In $D_{\ka}$, every variable appears only in the entries of a row; in particular,
 $z_j$ appears only in the entries of the $(j+\ka)$th row of $D_{\ka}$. Putting 
 \[m_{\ka}=l_0+l_1+\dotsb+l_{\ka-1},\]
 we get
 \begin{eqnarray*}
 \det(D_{\ka}(A_{\ka})) &=& \begin{vmatrix}
 							\om^{m_{\ka}\la_1}z_{m_{\ka}-\ka}& \om^{m_{\ka}\la_2}z_{m_{\ka}-\ka} &\hdots &\om^{m_{\ka}\la_{l_{\ka}}}z_{m_{\ka}-\ka}\\
 	\om^{(m_{\ka}+1)\la_1}z_{m_{\ka}+1-\ka}& \om^{(m_{\ka}+1)\la_2}z_{m_{\ka}+1-\ka} &\hdots &\om^{(m_{\ka}+1)\la_{l_{\ka}}}z_{m_{\ka}+1-\ka}\\
 	\vdots& \vdots& \ddots& \vdots\\
 	\om^{(m_{\ka+1}-1)\la_1}z_{m_{\ka+1}-1-\ka}& \om^{(m_{\ka+1}-1)\la_2}z_{m_{\ka+1}-1-\ka}& \hdots& \om^{(m_{\ka+1}-1)\la_{l_{\ka}}}z_{m_{\ka+1}-1-\ka}\\
 												\end{vmatrix}\\
 &=& \om^{m_{\ka}(\la_1+\dotsb+\la_{l_{\ka}})}z_{m_{\ka}-\ka}\dotsm z_{m_{\ka+1}-1-\ka}
 \begin{vmatrix}
 1& 1& \hdots &1\\
 \om^{\la_1}& \om^{\la_2}& \hdots &\om^{\la_{l_{\ka}}}\\
 \om^{2\la_1}& \om^{2\la_2}& \hdots &\om^{2\la_{l_{\ka}}}\\
 \vdots& \vdots& \ddots& \vdots\\
 \om^{(l_{\ka}-1)\la_1}& \om^{(l_{\ka}-1)\la_2}& \hdots &\om^{(l_{\ka}-1)\la_{l_{\ka}}}\\
 \end{vmatrix}\\
 &=& \om^{m_{\ka}(\la_1+\dotsb+\la_{l_{\ka}})}V(\om^{\la_1},\dotsc,\om^{\la_{l_{\ka}}})z_{m_{\ka}-\ka}\dotsm z_{m_{\ka+1}-1-\ka},
 \end{eqnarray*}
 where $V(x_1,x_2,\dotsc,x_n)$ is the standard Vandermonde determinant
 \[V(x_1,x_2,\dotsc,x_n)=\begin{vmatrix}
 1& 1& \hdots& 1\\
 x_1& x_2& \hdots& x_n\\
 x_1^2& x_2^2& \hdots& x_n^2\\
 \vdots& \vdots& \ddots& \vdots\\
 x_1^{n-1}& x_2^{n-1}& \hdots& x_n^{n-1}\\
 \end{vmatrix}
 =\prod_{1\leq i<j\leq n}(x_j-x_i).
 \]
 Since the $N$th roots of unity $\om^{\la_1},\dotsc,\om^{\la_{l_{\ka}}}$ are all distinct, the above Vandermonde determinant is
 nonvanishing; this is true for all $\ka$, thus we conclude that the coefficient of this CI monomial is nonzero.
 \end{proof}
 
 The only thing that remains to show now is that the CI monomial appears uniquely in $D$. Before proceeding to the proof of this
 statement, we will need to describe the equality case of the rearrangement inequality.
 
 \begin{lemma}\label{equal}
  Let $a_0<a_1<\dotsc<a_{N-1}$ and 
 \[b_0=b_1=\dotsb=b_{l_0-1}<b_{l_0}=\dotsb=b_{l_0+l_1-1}<\dotsb<b_{l_0+l_1+\dotsb+l_{N-2}}=\dotsb=b_{N-1}\]
 be real numbers. In other words, the increasing sequence $b_n$ is constant precisely
 on the intervals $A_0,A_1,\dotsc,A_{N-1}$. Then,
 \begin{equation}\label{rearr}
   \sum_{n=0}^{N-1}a_nb_n\geq\sum_{n=0}^{N-1} a_{\sig(n)}b_n
 \end{equation}
 for any permutation $\sig$ of the set $\set{1,2,\dotsc,N}$. 
  Equality occurs if and only if $\sig$ leaves the intervals $A_0,A_1,\dotsc,A_{N-1}$ invariant, i.~e. when
  $\sig$ is trivial.
 \end{lemma}
 
 \begin{proof}
  The inequality \eqref{rearr} is well-known and a proof is included in \cite{HLP34}, Chapter X. Furthermore, Theorem
  368 in \cite{HLP34} implies that
  \[\sum_{n=0}^{N-1} a_{\sig(n)}b_n\]
  attains its maximal value precisely when the finite sequences $a_{\sig(n)}$ and $b_n$ are simlarly ordered. When
  $\sig$ is trivial, equality is obvious in \eqref{rearr}. Suppose that $\sig$ is nontrivial, and let $\ka$
  be the minimal index such that $\sig(A_{\ka})\neq A_{\ka}$. Since $\sig$ leaves the sets $A_0,\dotsc,A_{\ka-1}$
  invariant and $A_{\ka}\cap\sig(A_{\ka})\neq A_{\ka}$, there are indices $\la,\mu>\ka$, such that
  $A_{\ka}\cap\sig(A_{\la})\neq\vn$ and $\sig(A_{\ka})\cap A_{\mu}\neq\vn$. Next, let $m$, $n$ be such that
  $\sig(n)\in A_{\ka}\cap\sig(A_{\la})$ and $\sig(m)\in\sig(A_{\ka})\cap A_{\mu}$. This implies that
  $b_n>b_m$, but $a_{\sig(n)}<a_{\sig(m)}$, so the sequences $a_{\sig(n)}$ and $b_n$ are not similarly ordered,
  therefore we have strict inequality in \eqref{rearr}.
 \end{proof}

 \section{Random variables associated to monomials}\label{RV}

 In section \ref{CI}, we saw that any permutation $\sig\in S_N$ gives rise to an ordered partition $\sig(A_0),\dotsc,\sig(A_{N-1})$
 of $\ZZ/N\ZZ$, which in turn gives rise to a monomial of $\det(D)$, say
 \[Z^{\sig}=z_0^{\al_0}z_1^{\al_1}\dotsm z_{N-1}^{\al_{N-1}},\]
 where $\al_i$ are nonnegative integers, with $\al_0+\al_1+\dotsb+\al_{N-1}=N$. To this monomial, we associate the discrete
 random variable $X_{\sig}$, which satisfies 
 \[P[X_{\sig}=i]=\frac{\al_{i}}{N}.\]
 So, to any $\sig\in S_N$, we associate a discrete random variable, $X_{\sig}$, and we shall say that such a random variable
 is obtained uniquely if $X_{\sig}=X_{\sig\tau}$ if and only if $\tau$ is trivial. We denote the random variable
 associated to the CI monomial by $X$. In order to complete the
 proof of Theorem 1.1, we need to show $X$ is obtained
 uniquely. The variables in the monomial that appears in $\det(D_{\ka}(A_{\ka}))$ have consecutive indices,
 in particular, the indices form the set 
 \[A_{\ka}-\ka=\set{m_{\ka}-\ka,m_{\ka}-\ka+1,\dotsc,m_{\ka+1}-(\ka+1)}=[m_{\ka}-\ka,m_{\ka+1}-(\ka+1)]\]
where we put
\[m_{\ka}= l_0+l_1+\dotsb+l_{\ka-1},\]
as before, and $m_0=0$. If $A_{\ka}=\vn$, then $m_{\ka}-\ka>m_{\ka+1}-(\ka+1)$, and $A_{\ka}-\ka=\vn$.
Now consider the $A_{\ka}-\ka$ as sets of \emph{integers}, rather than residues $\bmod N$.
\begin{prop}\label{interval}
 With notation as above, we have
 \[\bigcup_{\ka=0}^{N-1}(A_{\ka}-\ka)=[\al,\be],\]
 for some integers $\al$, $\be$.
\end{prop}

\begin{proof}
 If we put 
 \[\al=\min_{0\leq \ka\leq N-1}(m_{\ka}-\ka),\	\be=\max_{0\leq \ka\leq N-1}(m_{\ka}-\ka),\]
 then we obviously have
 \[\bigcup_{\ka=0}^{N-1}(A_{\ka}-\ka)\ssq[\al,\be].\]
 Now let $\de\in[\al,\be]$ be arbitrary. Suppose first that $\de=\al$. If $\al=0$, then $\de\in A_0=[m_0,m_1-1]$. If
 $\al<0$, let $\ka$ be an index such that $m_{\ka}-\ka=\al$. Obviously, $\ka<N$, and 
 $\de\in A_{\ka}-\ka=[m_{\ka}-\ka,m_{\ka+1}-(\ka+1)]$. Next, let $\de>\al$. If $\de\leq0$, let $\ka$ be the maximal
 index such that $m_{\ka}-\ka=\al$ (since $m_N-N=0$ and $\al<0$, we must have $\ka<N$), and let $\la$ be the minimal
 index satisfying $\ka<\la\leq N$ and $m_{\la}-\la\geq\de$. This shows that $\de\in A_{\la-1}-(\la-1)$. Lastly, if
 $\de>0$, take $\ka$ to be the minimal index satisfying $m_{\ka+1}-(\ka+1)\geq\de$ (since $m_0-0<\de$, we must
 have $\ka\geq0$); then, $\de\in A_{\ka}-\ka$.
 
 In every case, we have proven that any $\de$ in $[\al,\be]$, belongs to some set $A_{\ka}-\ka$. This 
 establishes the reverse inclusion as well, thus completing the proof.
\end{proof}

Since $\sum_{\ka=0}^{N-1}\abs{A_{\ka}-\ka}=N$, we will have $\be-\al\leq N-1$. Next, we will show that if we
translate the set $\La$, then the corresponding polynomials that we obtain as $\det(D)$ are essentially the same.

\begin{lemma}\label{translation}
 Let $\La\ssq(\ZZ/N\ZZ)^2$ with $\abs{\La}=N$, and let $D$ be a $N\times N$ matrix whose columns are $\pi(\ka,\la)z$,
 for $(\ka,\la)\in\La$. If $\La'$ is a translation of $\La$, and $D'$ the corresponding matrix, then
 \[\det(D')=c\det(D),\]
 for some nonzero $c$.
\end{lemma}

\begin{proof}
 It suffices to consider translations under vectors of the form $(\ga,0)$ or $(0,\ga)$. Suppose first that
 $\La'=\La-(\ga,0)$. Then, the columns of $D'$ have the form
 \[M^{\la}T^{\ka-\ga}z,\]
 as $(\ka,\la)$ runs throught the elements of $\La$. But since $MT=\om TM$, then 
 \[M^{\la}T^{\ka-\ga}z=\om^{-\la\ga}T^{-\ga}M^{\la}T^{\ka}z,\]
 hence
 \[D'=\om^{-\ga\sum_{(\ka,\la)\in\La}\la}E_{-\ga}M,\]
 where $E_{-\ga}$ is the permutation matrix, that moves the $j$ row to the $j+\ga$ row. Since $\det(E_{-\ga})=\pm1$,
 we get that $\det(D')=c\det(D)$, for some nonzero $c$.
 
 If $\La'=\La-(0,\ga)$, then the columns of $D'$ are
 \[M^{-\ga}M^{\la}T^{\ka}z,\]
 where $(\ka,\la)$ runs through the elements of $\La$. Eventually, we deduce that $D'$ is obtained by $D$, by
 multiplying the $j$ row of $D$ by $\om^{-j\ga}$, and thus we arrive to the same conclusion.
\end{proof}

An immediate consequence is that $\det(D)$ is nonzero if and only if $\det(D')$ is nonzero. Furthermore,
it is evident that 
the polynomial $\det(D')$ is obtained from $\det(D)$ by the following cyclic shift of the variables:
\[z_j \longrightarrow z_{j+\ga}.\]
We wish to translate $\La$ in order to obtain a new matrix with $\al=0$, where
\[\al=\min_{0\leq\ka\leq N-1}(m_{\ka}-\ka).\]
If $\al=0$, we do not need to translate $\La$; however, if $\al<0$, then
let $\ga$ be such that $\al=m_{\ga}-\ga$, and consider the translated set $\La'=\La-(\ga,0)$, obtaining a new
matrix $D'$, whose columns are described in the proof of Lemma \ref{translation}.

%
If
$l_{\ka}'$ denotes the columns of $D'$ of the form $\pi(\ka,\la)z$, then we have
\[l_{\ka}'=l_{\ka+\ga},\]
where the indices are considered as residues $\bmod N$. Define
\[m_{\ka}'=l_0'+l_1'+\dotsb+l_{\ka-1}'.\]
If $\ga+\ka-1<N$ (as integers), then
\begin{eqnarray*}
 m_{\ka}'-\ka &=& l_0'+l_1'+\dotsb+l_{\ka-1}'-\ka\\
 &=& l_{\ga}+l_{\ga+1}+\dotsb+l_{\ga+\ka-1}-\ka\\
 &=& m_{\ga+\ka}-m_{\ga}-\ka\\
 &=& (m_{\ga+\ka}-(\ga+\ka))-(m_{\ga}-\ga)\\
 &\geq& \al-\al=0,
\end{eqnarray*}
and if $\ga+\ka-1\geq N$, then
\begin{eqnarray*}
 m_{\ka}'-\ka &=& l_0'+l_1'+\dotsb+l_{\ka-1}'-\ka\\
 &=& l_{\la}+l_{\la+1}+\dotsb+l_{N-1}+l_0+\dotsb+l_{\ga+\ka-1-N}-\ka\\
 &=& N-m_{\ga}+m_{\ga+\ka-N}-\ka\\
 &=& (m_{\ga+\ka-N}-(\ga+\ka-N))-(m_{\ga}-\ga)\\
 &\geq& \al-\al=0,
\end{eqnarray*}
so we see that $m_{\ka}'-\ka\geq0$ for all $\ka$. Since $\det(D')$ and $\det(D)$ are the same polynomials up
to a nonzero multiplicative constant by Lemma \ref{translation}, we may assume without loss of generality that $m_{\ka}-\ka\geq0$
for all $\ka$. It turns out that the random variable associated to the CI monomial exhibits some unique
statistical properties, related to the other variables, $X_{\sig}$.



\begin{thm}\label{EV}
 Assuming that $D$ satisfies $m_{\ka}-\ka\geq0$ for all $\ka$, we have $E[X]\leq E[X_{\sig}]$
 for all permutations $\sig$. Furthermore, $E[X^2]\leq E[X_{\sig}^2]$, 
 with equality if and only if $\sig$ is trivial.
\end{thm}

\begin{proof}
 Define the sequence $\set{b_n}_{n=0}^{N-1}$ as follows:
 \[b_n=\ka,\text{ if }n\in A_{\ka}.\]
 From the definition of the partition $A_0,\dotsc,A_{N-1}$, it is clear that $b_n$ is increasing, and it is
 constant precisely on the intervals of integers $A_0,\dotsc,A_{N-1}$. All the indices that appear in the
 CI monomial are $n-b_n$, counting multiplicities; since $m_{\ka}-\ka\geq0$ for all $\ka$, if $n\in A_{\ka}$,
 $b_n=\ka$ and $n\geq m_{\ka}$, so $n-b_n\geq m_{\ka}-\ka\geq0$. The indices appearing in the monomial
 associated to the partition $\sig(A_0),\dotsc,\sig(A_{N-1})$ belong to the sets $\sig(A_{\ka})-\ka$,
 or equivalently, they have the form $\sig(n)-b_n$, \emph{as elements of} $\ZZ/N\ZZ$, counting multiplicities.
 Since $\abs{\sig(n)-b_n}<N$, it means that the index corresponding to $n$ is either $\sig(n)-b_n$
 or $\sig(n)-b_n+N$. In both cases, it is greater than or equal to $\sig(n)-b_n$, so
 \[E[X_{\sig}]\geq \frac{1}{N}\sum_{n=0}^{N-1}(\sig(n)-b_n)=\frac{1}{N}\sum_{n=0}^{N-1}(n-b_n)=E[X],\]
 which proves the first part of the theorem. For the second part, define
 \[\sig'(n)=\begin{cases}
             \sig(n), &\text{ if }\sig(n)-b_n\geq0\\
             \sig(n)+N, &\text{ if }\sig(n)-b_n<0,
            \end{cases}
\]
so that
\[E[X_{\sig}]=\frac{1}{N}\sum_{n=0}^{N-1}(\sig'(n)-b_n),\	\text{ and }
E[X_{\sig}^2]=\frac{1}{N}\sum_{n=0}^{N-1}(\sig'(n)-b_n)^2.\]
Let $C_1$ be the set of those $n$ for which $\sig(n)-b_n\geq0$, and $C_2$ be its complement in $\set{0,\dotsc,N-1}$
(we consider them as sets of integers). Also, let $m=\abs{C_1}$, and define 
\[f:\set{0,\dotsc,N-1}\longrar\sig(C_1)\cup(\sig(C_2)+N)\]
to be the unique strictly increasing function from $\set{0,\dotsc,N-1}$ to $\sig(C_1)\cup(\sig(C_2)+N)$. So, 
$f$ satisfies
\[f([0,m-1])=\sig(C_1),\	f([m,N-1])=\sig(C_2)+N,\	\text{ and }f(n)\geq n,\text{ for all }n.\]
Also, $\sig(C_1)\cup(\sig(C_2)+N)$ is the range of $\sig'$, so there is a permutation of $\sig(C_1)\cup(\sig(C_2)+N)$,
say $\tau$, such that
\[\sig'(n)=\tau(f(n)),\]
for all $n$. Since $f(n)\geq n$, we also have $f(n)-b_n\geq n-b_n\geq0$ for all $n$, so
\[E[X^2]\leq\frac{1}{N}\sum_{n=0}^{N-1}(f(n)-b_n)^2.\]
Next, we get
 \begin{eqnarray*}
  E[X_{\sig}^2]-\frac{1}{N}\sum_{n=0}^{N-1}(f(n)-b_n)^2 &=&
  \frac{1}{N}\sum_{n=0}^{N-1}(\sig'(n)-b_n)^2-\frac{1}{N}\sum_{n=0}^{N-1}(f(n)-b_n)^2\\
  &=& \frac{1}{N}\sum_{n=0}^{N-1}(\tau(f(n))-b_n)^2-\frac{1}{N}\sum_{n=0}^{N-1}(f(n)-b_n)^2\\
  &=& \frac{2}{N}\sum_{n=0}^{N-1}f(n)b_n-\frac{2}{N}\sum_{n=0}^{N-1}\tau(f(n))b_n\\
  &\geq& 0,
 \end{eqnarray*}
 by \eqref{rearr}, so eventually
 \begin{equation}\label{squares}
 E[X^2]\leq \frac{1}{N}\sum_{n=0}^{N-1}(f(n)-b_n)^2\leq E[X_{\sig}^2].
 \end{equation}
 When $C_2\neq\vn$, then for $n\in C_2$ we have $f(n)-b_n>n-b_n$, and we get a strict inequality in the
 left-hand side of \eqref{squares}, so if $E[X^2]=E[X_{\sig}^2]$, then $C_2=\vn$ and $f(n)=n$ for all $n$.
 Moreover, $\sig(n)-b_n\geq0$, for all $n$, so 
 \begin{eqnarray*}
 E[X_{\sig}^2]-E[X^2] &=& \frac{1}{N}\sum_{n=0}^{N-1}(\sig(n)-b_n)^2-\frac{1}{N}\sum_{n=0}^{N-1}(n-b_n)^2\\
  &=& \frac{2}{N}\sum_{n=0}^{N-1}nb_n-\frac{2}{N}\sum_{n=0}^{N-1}\sig(n)b_n\\
  &\geq& 0,
 \end{eqnarray*}
 with equality if and only if $\sig$ is trivial by Lemma \ref{equal}, completing the proof.
 \end{proof}

 \bigskip

\begin{proof}[Proof of Theorem \ref{main}]
We assume without loss of
generality that $m_{\ka}-\ka\geq0$ for all $\ka$.
 From Theorem \ref{EV} we deduce that the CI monomial can be obtained only through the partition $A_0,\dotsc,A_{N-1}$.
 Indeed, if we assume that the CI monomial is also obtained by some nontrivial partition
 $\sig(A_0),\dotsc,\sig(A_{N-1})$, this would show that $X=X_{\sig}$. But this contradicts Theorem \ref{EV}, because
 if $\sig$ nontrivial, then $E[X^2]<E[X_{\sig}^2]$, so these two random variables
 cannot be the same. So, the CI monomial is always obtained uniquely, therefore by Proposition \ref{unique}
 its coefficient in $\det(D)$ is nonzero, hence by the virtue of Lemma above, $\det(D)$ is a nonzero polynomial
 for any choice of
 $\La\ssq(\ZZ/N\ZZ)^2$, with $\abs{\La}=N$. This concludes the proof that Gabor frames exist in general
 linear position, in all dimensions; furthermore, since the zero set of the polynomials $\det(D)$ has Lebesgue
 measure zero, and these polynomials are finitely many, we deduce that the set of vectors generating a Gabor
 frame in general linear position is of full measure.
\end{proof}

\section{Construction}\label{construct}

We continue to use the same notation; we fix the matrix $D$, whose columns have the form $\pi(\ka,\la)z$, 
where $(\ka,\la)\in\La$, for $\abs{\La}=N$. Define $P_{\La}(z)=\det(M)$.


\begin{lemma}
 Define the polynomial $Q_{\La}(x)\in\QQ(\om)[x]$ by
 \[Q_{\La}(x)=P_{\La}(1,x,x^4,x^9,\dotsc,x^{(N-1)^2}).\]
 Then, $Q_{\La}$ is a nonzero polynomial, for all $\La\ssq(\ZZ/N\ZZ)^2$ with $\abs{\La}=N$.
\end{lemma}

\begin{proof}
 Again, by Lemma \ref{translation}, we may assume without loss of generality that $m_{\ka}-\ka\geq0$, for all $\ka$. Under the
 substitution $z_n=x^{n^2}$, the monomial
 \[Z^{\sig}=z_0^{\al_0}z_1^{\al_1}\dotsm z_{N-1}^{\al_{N-1}},\]
 becomes
 \[x^{\sum_{n=0}^{N-1}n^2\al_n}=x^{N\cdot E[X_{\sig}^2]},\]
 so if $P_{\La}(z)=\det(D)=\sum_{\sig\in S_N/\Ga}c_{\sig}Z^{\sig}$, then
 \[Q_{\La}(x)=\sum_{\sig\in S_N/\Ga}c_{\sig}x^{N\cdot E[X_{\sig}^2]}.\]
 The coefficient $c=c_{\iota}$ that corresponds to the CI monomial is nonzero, as Proposition \ref{CIM} implies,
 and Theorem \ref{EV} yields
 \[x^{N\cdot E[X^2]}\neq x^{N\cdot E[X_{\sig}^2]}\]
 for all $\sig\notin\Ga$, so $Q_{\La}(x)$ has a nonzero monomial, thus $Q_{\La}$ is a nonzero polynomial.
\end{proof}

Next, we observe that $\deg Q_{\La}\leq N(N-1)^2$, for all $\La$. Therefore:

\begin{cor}\label{unimodular}
 Let $\xi\in\CC$ be either a transcedental number, or an algebraic number whose degree over $\QQ(\om)$ is at
 least $N(N-1)^2+1$. Then, the vector
 \[(1,\xi,\xi^4,\xi^9,\dotsc,\xi^{(N-1)^2}),\]
 generates a Gabor frame in general linear position.
\end{cor}

\begin{proof}
 It suffices to prove that $Q_{\La}(\xi)\neq0$. But this follows from the fact that $\deg Q_{\La}\leq N(N-1)^2$;
 by hypothesis $\xi$ cannot be the root of any nonzero polynomial in $\QQ(\om)[x]$, whose degree is at
 most $N(N-1)^2$.
\end{proof}

It is evident that there is an abundance of such numbers $\xi$. We could put, for example, $\xi=\pi$, or $\xi=e$.
However, as $N$ is expected to be very large, it would be optimal to control the absolute value of all
coordinates of the above vector, by taking $\xi$ to be a root of unity.

\begin{cor}
 Let $\ze=e^{2\pi i/(N-1)^4}$, or any other primitive root of unity of order $(N-1)^4$, where $N\geq4$.
 Then, the vector
 \[(1,\ze,\ze^4,\ze^9,\dotsc,\ze^{(N-1)^2}),\]
 generates a Gabor frame in general linear position.
\end{cor}

\begin{proof}
 Since $\gcd(N,(N-1)^4)=1$, the degree of $\ze$ over $\QQ(\om)$ is the same as the degree of $\ze$ over $\QQ$,
 which is $\vphi((N-1)^4)=(N-1)^3\vphi(N-1)$. When $N\geq4$, we have
 \[(N-1)^3\vphi(N-1)\geq2(N-1)^3>N(N-1)^2,\]
 therefore by Corollary \ref{unimodular} 
 the vector
 \[(1,\ze,\ze^4,\ze^9,\dotsc,\ze^{(N-1)^2}),\]
 generates a Gabor frame in general linear position.
\end{proof}

\end{document}